\documentclass[12pt,leqno]{amsart}
\usepackage{amscd}
\usepackage{color}
\usepackage{amssymb}
\usepackage{amsfonts}
\usepackage{latexsym}
\usepackage{verbatim}

\theoremstyle{plain}
\newtheorem{theorem}{Theorem}[section]

\newtheorem{lemma}[theorem]{Lemma}
\newtheorem{prop}[theorem]{Proposition}

\renewcommand{\b}{\begin{equation}}
\newcommand{\e}{\end{equation}}

\newcommand\R{{\mathbb R}}

\title[Calabi Estimate]{On the Calabi estimate of geometric flows of Hermitian metrics }

\makeatletter
\@namedef{subjclassname@2020}{%
 \textup{2020} Mathematics Subject Classification}
\makeatother
\thanks{This work was supported by GNSAGA of INdAM}

\address{ Dipartimento di Matematica G. Peano \\ Universit\`a di Torino\\
Via Carlo Alberto 10\\
10123 Torino\\ Italy}
\email{ma.gallo@unito.it, luigi.vezzoni@unito.it}

\begin{document}
\author{Marco Gallo and Luigi Vezzoni}
\maketitle

\date{\today}
\begin{abstract}  
We establish a general result ensuring a $C^1$
 a priori bound for smooth curves of Hermitian metrics. As a main application, we obtain a new regularity result for Hermitian curvature flows, and in particular for the second Chern–Ricci flow.     
\end{abstract}

\section{Introduction}
In the present paper we focus on geometric flows of Hermitian metrics in compact complex manifolds. After Cao proved \cite{Cao} that the Ricci flow on K\"ahler manifolds could be used to provide an alternative proof of the Calabi-Yau theorem  \cite{Yau}, the K\"ahler-Ricci flow became a fundamental tool in K\"ahler geometry. In contrast to the K\"ahler case, in the Hermitian setting with torsion the Ricci flow does not in general preserve the Hermitian property of the initial metric and its role is often replaced by other second order geometric flows. In \cite{HCF} Streets and Tian introduced a family of geometric flows of Hermitian metrics (Hermitian curvature flows, or shortly ${\rm HCF}$) which evolves an initial Hermitian metric by its second Chern-Ricci form plus a quadric term $Q$ in the torsion. The following belong to this family: the {\em gradient HCF} \cite{HCF}, the {\em pluriclosed flow} \cite{ST}, the {\em positive HCF} \cite{YU1} and the {\em second Chern-Ricci flow} \cite{Lee}. Another evolution equation was introduced by Tosatti and Weinkove in \cite{TW} who considered the geometric flow governed by the first Chern-Ricci form. In contrast to the Hermitian curvature flows, the {\em Chern-Ricci flow} is a potential flow and it can be reduced to a parabolic complex Monge-Amp\`ere equation. 
A different approach was considered by 
Phong, Picard and Zhang with the anomaly flow \cite{AF8} and the type IIB flow \cite{AF7}. These flows are not a generalization of the K\"ahler-Ricci flow and their study is originated by the study of the Hull-Strominger system \cite{hull,Strominger}. 

\medskip 
We say that a smooth curve of Hermitian metrics $g=g(t)$ on a compact Hermitian manifold $(M,\hat g)$ satisfies a {\em Calabi estimate} if 
\begin{equation}\label{Calabi}
|\hat \nabla g|_{g}^2\leq C
\end{equation}
for a uniform constant $C$, where $\hat \nabla$ is the Chern connection of $\hat g$. Usually to have the Calabi estimate on a solution $g$ to a geometric flow, we have to impose that $g$ and $\hat g$ satisfy the {\em uniform equivalence relation}  
\begin{equation}\label{ass1}
K^{-1}\hat g\leq g\leq K\hat g
\end{equation}
for a uniform constant $K$. A Calabi-type estimate holds for the K\"ahler-Ricci flow \cite{SW1}, the Chern-Ricci flow \cite{SW2}, the pluriclosed flow \cite{mario,jordan,S2} and for the type IIB flow under a control on the trace of the torsion \cite{estimates}. If $g$ solves a quasi-linear parabolic equation, then the Calabi estimate combined with the conformal equivalence \eqref{ass1} implies a control on the high-order derivatives of $g$ (see \cite[Section 4]{estimates}). This can be in particular applied to Hermitian curvature flows \cite{estimates}.  

\medskip 
The aim of the present paper is to provide a unified point of view by pointing out general sufficient conditions for a geometric flow to satisfy a Calabi estimate \eqref{Calabi}. 
Some of the results we mentioned before can be deduced from our results.

\medskip 
Let $(M,\hat g)$ be a compact Hermitian manifold and let $g=g(t)$ be a smooth curve  of Hermitian metrics on $M$.  Let $\hat \nabla$ and $\nabla$ be the Chern connections of $\hat  g$ and $g$, respectively, and let  
$$
S:=|\hat\nabla g|_{g}^2=|\Psi|_g^2\,,
$$
where $\Psi:=\nabla -\hat \nabla$.  
     
\begin{theorem}\label{main1}
Let $g=g(t)$, be a smooth curve of Hermitian metrics on a compact Hermitian manifold $(M,\hat g)$. Assume that $g$ satisfies the uniform equivalent condition \\ \eqref{ass1} and  
\begin{eqnarray}
&&\label{ass2} |\dot g+\widetilde{\rm Ric}|_g\leq K (1+S^{1/2})\,,\\
&&\label{ass3} |\nabla(\dot g+\widetilde{\rm Ric})|_{g}\leq 
 K (|\nabla \Psi|_g+ |\bar \nabla \Psi|_g+1+S)\,,\\
&&\label{ass4} |T|_g\leq K
\end{eqnarray}
for a constant $K$. 
Then there exists a constant $C$ depending on $M$, $K$, $\hat g$ and $g(0)$ only, such that 
$$
|\hat \nabla g|_{g}^2\leq C\,. 
$$
\end{theorem}

In particular if a curve of Hermitian metrics $g$ satisfies the uniform estimate \eqref{ass1} and the quantities
 $|\dot g+\tilde R|_g$, $|\nabla(\dot g+\tilde R)|_{g}$, $|T|_g$ are bounded, then $g$ satisfies a uniform $C^1$-bound.  
 
 \medskip 
As first natural application of Theorem \ref{main1} we have the following result 
\begin{theorem}\label{ThHCF}
Let $(M,\hat g)$ be a compact Hermitian manifold and let $g=g(t)$, $t\in [0,\tau)$ be a maximal-time solution of a flow belonging to the family of the HCFs. Assume $\tau <\infty$. 
Then 
$$
\limsup_{t\to \tau}\,\max \{|{\rm tr}_{g}\hat g|_{C^0(g)},|T|_{C^0(g)}\}=\infty\,.
$$
\end{theorem}
Note that Theorem \ref{ThHCF} is consistent with \cite{S2} where it is shown that for the pluriclosed flow a bound on ${\rm tr}_{g}\hat g$ is enough to have the long time existence. 

\medskip 
The following result focuses on the second Chern-Ricci flow.  

\begin{theorem}\label{ThH=0}
Let $g=g(t)$, $t\in [0,\tau)$ be a maximal-time solution to the second Chern-Ricci flow 
$$
\dot g=-\widetilde{\rm Ric}
$$
on a compact complex manifold. 
If $\tau <\infty$, then the Chern curvature ${\rm Rm}$ of $g$ satisfies
$$
\limsup_{t\to \tau}|{\rm Rm}|_{C^0(g)}=\infty\,.
$$
\end{theorem}
An analogue result was proved in \cite{ST} for the pluriclosed flow on surfaces. Moreover, using the Gauge equivalence between the pluriclosed flow and the generalized Ricci flow \cite[Theorem 1.1]{ST4}, \cite[Theorem 5.23]{{mariobook}} implies that in any dimension the Riemannian curvature blows up at time singularities of the pluriclosed flow.  
Theorem \ref{ThH=0} is obtained by using \cite[Theorem 1.1]{HCF} and some general evolution equations proved in Section \ref{GFormulas}.   

\medskip 
A larger class of geometric flows is given by flows satisfying the following relation 
\begin{equation}\label{moregeneral}
\dot g+\widetilde{\rm Ric}=g*\bar \nabla T+T*T\,, 
\end{equation}
(see Section \ref{applications}).  Here given two tensors $A$ and $B$ we denote, as usual, by $A*B$ any tensor which is a linear combination
of tensors formed by starting with $A\otimes B$ and then contracting by using the metric $g$. 
Theorem \ref{main1} implies a Calabi estimate for flows satisfying  \eqref{moregeneral}  when we impose a better control on the torsion: 

\begin{prop}\label{2cor}
Let $g=g(t)$ be a smooth curve of Hermitian metrics on a compact Hermitian manifold $(M,\hat g)$. Assume that  $g$ satisfies \eqref{ass1}, \eqref{moregeneral} and 
\begin{equation}\label{Tcond}
|T|_g\leq K\,,\quad \left|\bar{\hat \nabla} T \right|_g\leq K\,,\quad \left|\hat \nabla\bar {\hat \nabla} T \right|_g\leq K
\end{equation}
for a uniform constant $K$. 
Then there exists a positive constant $C$ depending on $K$ and $\hat g$ such that 
$$
|\hat \nabla g|_{g}^2\leq C\,. 
$$
\end{prop}

Note that if the torsion of $g(t)$ is constant in time, then conditions \eqref{Tcond} are in particular satisfied and Corollary \ref{2cor} can be applied. This in particular happens when $d\dot\omega(t)=0$, as in the case of the Chern-Ricci flow.


\medskip 
\noindent 
\newline 
{\bf Acknowledgments:}
The authors would like to thank Jeffrey Streets for his helpful comments on an earlier version of the paper.

\section{Preliminaries}
Let $(M,g)$ be a Hermitian manifold. Fix local complex coordinates and denote by $g_{r\bar s}$ components of $g$. Let  
 $\nabla$ be the
Chern connection of $g$. Then the Christoffel symbols of $\nabla$ are given by 
$$
\Gamma_{ij}^k=g^{\bar rk}g_{j	\bar r,i}\,,
$$
and the torsion tensor $T$ of $\nabla$ locally reads as 
$$
T_{ij}^k=g^{\bar rk}g_{j	\bar r,i}-g^{\bar rk}g_{i\bar r,j}\,.
$$
Moreover, the components of the curvature tensor ${\rm Rm}$ of $\nabla$ are given by
\begin{equation*}\label{tildeR}
R_{k\bar l r	\bar s  }=-g_{u\bar s}\nabla_{\bar l}\Gamma_{kr}^{u}=-g_{r \bar s,k\bar l}+g^{\bar b a}g_{a\bar s,\bar l }g_{r	\bar b,k}
\end{equation*} 
and the components of the  
{\em first} and the {\em second Chern-Ricci tensor} ${\rm Ric}$ and $\widetilde{{\rm Ric}}$ are respectively
given by
$$
 R_{r\bar s}=g^{\bar lk}R_{r\bar sk\bar l }\,,\quad \mbox{ and }\quad  \tilde R_{r\bar s}=g^{\bar lk}R_{k\bar l r\bar s}\,.
 $$
In the non-K\"ahler case the two tensors ${\rm Ric}$ and $\widetilde{{\rm Ric}}$ do not in general agree and we have the following 
formula
\begin{equation}\label{R-R}
R_{i\bar j}-\tilde{R}_{i\bar j}=g^{\bar lk}\left(\nabla_{\bar l}T_{ki\bar j}-\nabla_iT_{\bar l\bar jk}\right)\,,
\end{equation}
see e.g. \cite[Lemma 2.4]{HCF}. 

\medskip 
We further recall the definition of some geometric flows of  Hermitian metrics studied in the literature and already mentioned in the introduction. The {\em Hermitian curvature flows} 
introduced by Streets and Tian are defined by the following expression
\begin{equation*}\label{HCF}
\partial_t g=-\widetilde{\rm Ric}+Q
\end{equation*}
where $Q$ is quadratic in the torsion of $g$. The flows are distinguished by the choice of $Q$ which, a priori, could be 
arbitrary. 
The following flows belong to this family:

\medskip \noindent 
the {\em pluriclosed flow} (PCF) \cite{ST}
$$
\partial_t g_{r	\bar s}=-\tilde R_{r	\bar s}+g^{\bar ba}g^{\bar lm}T_{ra	\bar l}T_{\bar s\bar b m };
$$
the {\em Positive Hermitian curvature flow} (HCF$_{+}$) \cite{YU1}
$$
\partial_t g_{r	\bar s}=-\tilde R_{r	\bar s}-\frac12 g^{\bar ba}g^{\bar lm} T_{a	 m\bar s}T_{\bar b \bar l r};
$$ 

\smallskip 
\noindent  the {\em Second  Chern-Ricci flow} \cite{Lee}
$$
\partial_t g_{r	\bar s}=-\tilde R_{r\bar s}\,. 
$$ 
The PCF preserves the pluriclosed condition $\partial \bar\partial \omega=0$, while the  {\rm HCF}$_+$ improve the positivity of the  holomorphic bisectional curvature \cite[Theorem 01]{YU1}. 

In \cite{TW} Tosatti and Weinkove introduced the {\em Chern-Ricci flow}
$$
\partial_t g=-{\rm Ric}\,,
$$ 
while in \cite{AF8} and \cite{AF7}  Phong, Picard and Zhang  introduced the {\em Type IIB flow} 
$$
\partial_t(|\Omega|_{\omega}\omega^{n-1})=i\partial \bar \partial \omega^{n-2}-\psi
$$
where $\Omega$ is a holomorphic volume form and 
$\psi \in \Lambda^{n-1,n-1}_{\mathbb R}$ is closed. When $\psi=0$ the flow is called {\em anomaly flow}. The flow preserves the {\em conformally balanced condition } 
$$
d(|\Omega|_{\omega}\omega^{n-1})=0
$$ and under this assumption the metric $g$ induced by the form $|\Omega|_{\omega}\,\omega$ solves 
\begin{equation}\label{typeIIB}
\partial_t g_{r\bar s}=-\tilde R_{r\bar s}-\frac12 g^{\bar ba}g^{\bar lm} T_{a	 m\bar s}T_{\bar b \bar l r}+\Phi_{r\bar s}\,,
\end{equation}
where $\Phi\in \Lambda^{1,1}_{\mathbb R}$ involves combinations of the components of $	\psi$ and $g$ \cite{AF9}. In particular when $\psi=0$, flow \eqref{typeIIB} is the positive anomaly flow.

\section{General Formulas}\label{GFormulas}
Let $g=g(t)$ be a smooth curve of Hermitian metric on a compact Hermitian manifold $(M,\hat g)$. Let $\nabla$ and $\hat \nabla$ be the Chern connection of $g$ and $\hat g$, respectively, and 
$$
S:=|\hat \nabla g|_{g}^2\,. 
$$
Let us denote by  
$$
\Psi=\nabla-\hat \nabla\,,
$$
in order to write 
$$
S=|\Psi|_{g}^2\,. 
$$
\begin{lemma}
We have 
$$
T_{ij}^k-\hat T_{ij}^k=\Psi_{ij}^k-\Psi_{ji}^k\,. 
$$
\end{lemma}
\begin{proof}
We directly compute 
$$
T_{ij}^k-\hat T_{ij}^k=\Gamma_{ij}^k-\Gamma_{ji}^k-\hat \Gamma_{ij}^k+\hat \Gamma_{ji}^k=\Psi_{ij}^k-\Psi_{ji}^k\,,
$$
and the claim follows. 
\end{proof}
Furthermore we denote by 
$$
\Delta:=g^{\bar s r}\partial_{r}\partial_{\bar s}
$$
the Chern Laplacian of $g$ and we set 
$$
\square:=\partial_t-\Delta\,.
$$
\begin{prop}\label{BoxS}
We have 
$$
\begin{aligned}
\Box S=& -|\nabla \Psi|_g^2-|\bar \nabla \Psi|_g^2+2{\rm Re}(g^{\bar mi}g^{\bar pj}\nabla_i(\dot g_{j\bar q}+\tilde R_{j\bar q})\Psi_{\bar m\bar p}^{\bar q})\\
&+g^{\bar mi}g^{\bar pj}\Psi_{ij}^k(\dot g_{k\bar q}+\tilde R_{k\bar q})\Psi_{\bar u\bar p}^{\bar q}
+g_{k\bar q}g^{\bar mi}g^{\bar pj}\Psi_{ij}^k(\dot g_{\bar p }^{\bar u}+\tilde R_{\bar p }^{\bar u})\Psi^{\bar q}_{\bar m\bar u}\\
&-g_{k\bar q}g^{\bar mi}g^{\bar pj}\Psi_{ij}^k(\dot g_{\bar u }^{\bar q}+\tilde R_{\bar u }^{\bar q})\Psi^{\bar u}_{\bar m\bar p}
+2{\rm Re}\,\left(g^{\bar mi}g^{\bar pj}g^{\bar ba}T_{ai}^k\nabla_{\bar b}\Psi_{kj\bar q}\Psi_{\bar m\bar p}^{\bar q}\right)\\
&-2 {\rm Re}\,\left(g^{\bar mi}g^{\bar pj}g^{\bar ba}T_{ai}^k\hat R_{k\bar bj\bar q}\Psi_{\bar m\bar p}^{\bar q}\right)-2{\rm Re}\left(g^{\bar mi}g^{\bar pj}g^{\bar ba}\nabla_a\hat R_{i\bar bj	\bar q}\Psi_{\bar m\bar p}^{\bar q}\right)\,.
\end{aligned}
$$
\end{prop}
\begin{proof}
We have
$$
\Psi_{ij}^k=\Gamma_{ij}^k-\hat \Gamma_{ij}^k\,,
$$
and
\begin{equation}\label{for}
\nabla_{\bar r}\Psi_{ij}^k=-R_{i\bar r j}^k+\hat R_{i\bar r j}^k\,,
\end{equation}
and 
$$
S=g_{k\bar q}g^{\bar mi}g^{\bar pj}\Psi_{ij}^k\Psi_{\bar m\bar p}^{\bar q}=g^{\bar mi}g^{\bar pj}g^{\bar r k}\Psi_{j\bar ri}\Psi_{k\bar p\bar m}\,.
$$
We directly compute
$$
\begin{aligned}
\partial_t S=&\dot g_{k\bar q}g^{\bar mi}g^{\bar pj}\Psi_{ij}^k \Psi_{\bar m\bar p}^{\bar q}
-g_{k\bar q}\dot g_{\bar b}^{\bar m}g^{\bar bi}g^{\bar pj}\Psi_{ij}^k \Psi_{\bar m\bar p}^{\bar q}
-g_{k\bar q}g^{\bar mi}\dot g_{\bar b}^{\bar p}g^{\bar bj}\Psi_{ij}^k \Psi_{\bar m\bar p}^{\bar q}\\
&+g_{k\bar q}g^{\bar mi}g^{\bar pj}\dot \Psi_{ij}^k\Psi_{\bar m\bar p}^{\bar q}+g_{k\bar q}g^{\bar mi}g^{\bar pj}\Psi_{ij}^k\dot \Psi_{\bar m\bar p}^{\bar q}\,.
\end{aligned}
$$
Since 
$$
\dot \Psi_{ij}^k=\partial_t\Gamma_{ij}^k=\partial_t(g^{\bar rk}g_{j	\bar r,i})=-g^{\bar ra}\dot g_{a\bar b}g^{\bar bk}g_{j	\bar r,i}+ g^{\bar rk}\dot g_{j	\bar r,i}=g^{\bar bk}(\Gamma_{ij}^a\dot g_{a\bar b}+g_{j	\bar b,i})=\nabla_i\dot g_{j}^k
$$
we can write 
$$
\begin{aligned}
\partial_t S=&\dot g_{k\bar q}g^{\bar mi}g^{\bar pj}\Psi_{ij}^k \Psi_{\bar m\bar p}^{\bar q}
-g_{k\bar q}\dot g_{\bar b}^{\bar m}g^{\bar bi}g^{\bar pj}\Psi_{ij}^k \Psi_{\bar m\bar p}^{\bar q}
-g_{k\bar q}g^{\bar mi}\dot g_{\bar b}^{\bar p}g^{\bar bj}\Psi_{ij}^k \Psi_{\bar m\bar p}^{\bar q}\\
&+g_{k\bar q}g^{\bar mi}g^{\bar pj}\nabla_i\dot g_{j}^k\Psi_{\bar m\bar p}^{\bar q}+g_{k\bar q}g^{\bar mi}g^{\bar pj}\Psi_{ij}^k\nabla_{\bar m}\dot g_{\bar p}^{\bar q}\,.
\end{aligned}
$$
Moreover 
$$
\nabla_{\bar b}S=g_{k\bar q}g^{\bar mi}g^{\bar pj}\nabla_{\bar b}\Psi_{ij}^k\Psi_{\bar m\bar p}^{\bar q}+g_{k\bar q}g^{\bar mi}g^{\bar pj}\Psi_{ij}^k\nabla_{\bar b}\Psi_{\bar m\bar p}^{\bar q}\,,
$$
and 
$$
\begin{aligned}
\Delta S=&\, g^{\bar ba}\nabla_a\nabla_{\bar b}S=
g_{k\bar q}g^{\bar mi}g^{\bar pj}g^{\bar ba}(\nabla_a\nabla_{\bar b}\Psi_{ij}^k)\Psi_{\bar m\bar p}^{\bar q}
+g_{k\bar q}g^{\bar mi}g^{\bar pj}\Psi_{ij}^kg^{\bar ba}(\nabla_a\nabla_{\bar b}\Psi_{\bar m\bar p}^{\bar q})\\
&\,+g_{k\bar q}g^{\bar mi}g^{\bar pj}\,g^{\bar ba}\nabla_{\bar b}\Psi_{ij}^k\nabla_a\Psi_{\bar m\bar p}^{\bar q}+g_{k\bar q}g^{\bar mi}g^{\bar pj}g^{\bar ba}\nabla_a\Psi_{ij}^k\nabla_{\bar b}\Psi_{\bar m\bar p}^{\bar q}\\
= &|\nabla \Psi|_g^2+|\bar \nabla \Psi|_g^2+g_{k\bar q}g^{\bar mi}g^{\bar pj}g^{\bar ba}(\nabla_a\nabla_{\bar b}\Psi_{ij}^k)
\Psi_{\bar m\bar p}^{\bar q}+g_{k\bar q}g^{\bar mi}g^{\bar pj}\Psi_{ij}^k g^{\bar ba}(\nabla_a\nabla_{\bar b}\Psi_{\bar m\bar p}^{\bar q})\,.
\end{aligned}
$$
Therefore using \eqref{for} we get 
$$
\begin{aligned}
\Delta S
=& \,|\nabla \Psi|_g^2+|\bar \nabla \Psi|_g^2
-g^{\bar mi}g^{\bar pj}g^{\bar ba}\nabla_aR_{i\bar bj	\bar q}\Psi_{\bar m\bar p}^{\bar q}-g^{\bar mi}g^{\bar pj}\Psi_{ij}^kg^{\bar ba}\nabla_{\bar b}R_{\bar ma\bar pk}\\
&-g_{k\bar q}g^{\bar mi}g^{\bar pj}\Psi_{ij}^kg^{\bar ba}R_{a\bar b\bar m }^{\bar u}\Psi_{\bar u\bar p}^{\bar q}
-g_{k\bar q}g^{\bar mi}g^{\bar pj}\Psi_{ij}^kg^{\bar ba}R_{a\bar b\bar p }^{\bar u}\Psi^{\bar q}_{\bar m\bar u}\\
&+g_{k\bar q}g^{\bar mi}g^{\bar pj}\Psi_{ij}^kg^{\bar ba}R_{a\bar b\bar u }^{\bar q}\Psi^{\bar u}_{\bar m\bar p}\\
&+g^{\bar mi}g^{\bar pj}g^{\bar ba}\nabla_a\hat R_{i\bar bj	\bar q}\Psi_{\bar m\bar p}^{\bar q}+g^{\bar mi}g^{\bar pj}\Psi_{ij}^kg^{\bar ba}\nabla_{\bar b}\hat R_{\bar ma\bar pk}\\
=& \,|\nabla \Psi|_g^2+|\bar \nabla \Psi|_g^2-g^{\bar mi}g^{\bar pj}g^{\bar ba}\nabla_aR_{i\bar bj	\bar q}\Psi_{\bar m\bar p}^{\bar q}
-g^{\bar mi}g^{\bar pj}\Psi_{ij}^kg^{\bar ba}\nabla_{\bar b}R_{\bar m a\bar pk}\\
&+g_{k\bar q}g^{\bar mi}g^{\bar pj}\Psi_{ij}^k\tilde R_{\bar m }^{\bar u}\Psi_{\bar u\bar p}^{\bar q}
+g_{k\bar q}g^{\bar mi}g^{\bar pj}\Psi_{ij}^k\tilde R_{\bar p }^{\bar u}\Psi^{\bar q}_{\bar m\bar u}
-g_{k\bar q}g^{\bar mi}g^{\bar pj}\Psi_{ij}^k\tilde R_{\bar u }^{\bar q}\Psi^{\bar u}_{\bar m\bar p}\\
&+g^{\bar mi}g^{\bar pj}g^{\bar ba}\nabla_a\hat R_{i\bar bj	\bar q}\Psi_{\bar m\bar p}^{\bar q}+g^{\bar mi}g^{\bar pj}\Psi_{ij}^kg^{\bar ba}\nabla_{\bar b}\hat R_{\bar ma\bar pk}
\end{aligned}
$$
and the second Bianchi identity implies 
$$
\begin{aligned}
\Delta S
=& \,|\nabla \Psi|_g^2+|\bar \nabla \Psi|_g^2-g^{\bar mi}g^{\bar pj}g^{\bar ba}\nabla_iR_{a\bar bj	\bar q}\Psi_{\bar m\bar p}^{\bar q}
+g^{\bar mi}g^{\bar pj}g^{\bar ba}T_{ai}^kR_{k\bar bj\bar q}\Psi_{\bar m\bar p}^{\bar q}\\
&-g^{\bar mi}g^{\bar pj}\Psi_{ij}^kg^{\bar ba}\nabla_{\bar m}R_{a \bar b k \bar p}+g^{\bar mi}g^{\bar pj}g^{\bar ba}T_{\bar b\bar m}^{\bar s}R_{\bar s a \bar p k}\Psi_{ij}^k\\
&+g_{k\bar q}g^{\bar mi}g^{\bar pj}\Psi_{ij}^k\tilde R_{\bar m }^{\bar u}\Psi_{\bar u\bar p}^{\bar q}
+g_{k\bar q}g^{\bar mi}g^{\bar pj}\Psi_{ij}^k\tilde R_{\bar p }^{\bar u}\Psi^{\bar q}_{\bar m\bar u}
-g_{k\bar q}g^{\bar mi}g^{\bar pj}\Psi_{ij}^k\tilde R_{\bar u }^{\bar q}\Psi^{\bar u}_{\bar m\bar p}\\
&+g^{\bar mi}g^{\bar pj}g^{\bar ba}\nabla_a\hat R_{i\bar bj	\bar q}\Psi_{\bar m\bar p}^{\bar q}+g^{\bar mi}g^{\bar pj}\Psi_{ij}^kg^{\bar ba}\nabla_{\bar b}\hat R_{\bar ma\bar pk}
\end{aligned}
$$
and 
$$
\begin{aligned}
\Delta S
=& \,|\nabla \Psi|_g^2+|\bar \nabla \Psi|_g^2-g^{\bar mi}g^{\bar pj}\nabla_i\tilde R_{j\bar q}\Psi_{\bar m\bar p}^{\bar q}
+g^{\bar mi}g^{\bar pj}g^{\bar ba}T_{ai}^kR_{k\bar bj\bar q}\Psi_{\bar m\bar p}^{\bar q}\\
&-g^{\bar mi}g^{\bar pj}\Psi_{ij}^k\nabla_{\bar m}\tilde R_{k \bar p}+g^{\bar mi}g^{\bar pj}g^{\bar ba}T_{\bar b\bar m}^{\bar s}R_{\bar s a \bar p k}\Psi_{ij}^k\\
&+g_{k\bar q}g^{\bar mi}g^{\bar pj}\Psi_{ij}^k\tilde R_{\bar m }^{\bar u}\Psi_{\bar u\bar p}^{\bar q}
+g_{k\bar q}g^{\bar mi}g^{\bar pj}\Psi_{ij}^k\tilde R_{\bar p }^{\bar u}\Psi^{\bar q}_{\bar m\bar u}
-g^{\bar mi}g^{\bar pj}\Psi_{ij}^k\tilde R_{k\bar u }\Psi^{\bar u}_{\bar m\bar p}\\
&+g^{\bar mi}g^{\bar pj}g^{\bar ba}\nabla_a\hat R_{i\bar bj	\bar q}\Psi_{\bar m\bar p}^{\bar q}+g^{\bar mi}g^{\bar pj}\Psi_{ij}^kg^{\bar ba}\nabla_{\bar b}\hat R_{\bar ma\bar pk}\,.
\end{aligned}
$$
Therefore 
$$
\begin{aligned}
\Box S=& -|\nabla \Psi|_g^2-|\bar \nabla \Psi|_g^2+g^{\bar mi}g^{\bar pj}\nabla_i(\dot g_{j\bar q}+\tilde R_{j\bar q})\Psi_{\bar m\bar p}^{\bar q}+g^{\bar mi}g^{\bar pj}\Psi_{ij}^k\nabla_{\bar m}(\dot g_{k\bar p}+\tilde R_{k\bar p})\\
&+g^{\bar mi}g^{\bar pj}\Psi_{ij}^k(\dot g_{k\bar q}+\tilde R_{k\bar q})\Psi_{\bar u\bar p}^{\bar q}
-g_{k\bar q}g^{\bar mi}g^{\bar pj}\Psi_{ij}^k(\dot g_{\bar p }^{\bar u}+\tilde R_{\bar p }^{\bar u})\Psi^{\bar q}_{\bar m\bar u}\\
&-g_{k\bar q}g^{\bar mi}g^{\bar pj}\Psi_{ij}^k(\dot g_{\bar u }^{\bar q}+\tilde R_{\bar u }^{\bar q})\Psi^{\bar u}_{\bar m\bar p}\\
&
-g^{\bar mi}g^{\bar pj}g^{\bar ba}T_{ai}^kR_{k\bar bj\bar q}\Psi_{\bar m\bar p}^{\bar q}
-g^{\bar mi}g^{\bar pj}g^{\bar ba}T_{\bar b\bar m}^{\bar s}R_{\bar s a \bar p k}\Psi_{ij}^k\\
&-g^{\bar mi}g^{\bar pj}g^{\bar ba}\nabla_a\hat R_{i\bar bj	\bar q}\Psi_{\bar m\bar p}^{\bar q}-g^{\bar mi}g^{\bar pj}\Psi_{ij}^kg^{\bar ba}\nabla_{\bar b}\hat R_{\bar ma\bar pk}
\end{aligned}
$$
and using \eqref{for} we conclude 
$$
\begin{aligned}
\Box S=& -|\nabla \Psi|_g^2-|\bar \nabla \Psi|_g^2+g^{\bar mi}g^{\bar pj}\nabla_i(\dot g_{j\bar q}+\tilde R_{j\bar q})\Psi_{\bar m\bar p}^{\bar q}+g^{\bar mi}g^{\bar pj}\Psi_{ij}^k\nabla_{\bar m}(\dot g_{k\bar p}+\tilde R_{k\bar p})\\
&+g^{\bar mi}g^{\bar pj}\Psi_{ij}^k(\dot g_{k\bar q}+\tilde R_{k\bar q})\Psi_{\bar u\bar p}^{\bar q}
-g_{k\bar q}g^{\bar mi}g^{\bar pj}\Psi_{ij}^k(\dot g_{\bar p }^{\bar u}+\tilde R_{\bar p }^{\bar u})\Psi^{\bar q}_{\bar m\bar u}\\
&-g_{k\bar q}g^{\bar mi}g^{\bar pj}\Psi_{ij}^k(\dot g_{\bar u }^{\bar q}+\tilde R_{\bar u }^{\bar q})\Psi^{\bar u}_{\bar m\bar p}\\
&
+g^{\bar mi}g^{\bar pj}g^{\bar ba}T_{ai}^k\nabla_{\bar b}\Psi_{kj\bar q}\Psi_{\bar m\bar p}^{\bar q}
+g^{\bar mi}g^{\bar pj}g^{\bar ba}T_{\bar b\bar m}^{\bar s}\nabla_{a}\Psi_{\bar s \bar p k}\Psi_{ij}^k\\
&-g^{\bar mi}g^{\bar pj}g^{\bar ba}T_{ai}^k\hat R_{k\bar bj\bar q}\Psi_{\bar m\bar p}^{\bar q}
-g^{\bar mi}g^{\bar pj}g^{\bar ba}T_{\bar b\bar m}^{\bar s}\hat R_{\bar s a \bar p k}\Psi_{ij}^k\\
&-g^{\bar mi}g^{\bar pj}g^{\bar ba}\nabla_a\hat R_{i\bar bj	\bar q}\Psi_{\bar m\bar p}^{\bar q}-g^{\bar mi}g^{\bar pj}\Psi_{ij}^kg^{\bar ba}\nabla_{\bar b}\hat R_{\bar ma\bar pk}\,,
\end{aligned}
$$
as required. 
\end{proof}
\begin{lemma}\label{lemma1}
We have 
$$
\Box  {\rm tr}_{\hat g} g=
\hat g^{\bar s r}(\dot g_{r \bar s}+\tilde R_{r \bar s})-\hat g^{\bar s r}g^{\bar lk}g^{\bar b a}\hat \nabla_{\bar l}g_{a\bar s}\hat \nabla_{k}g_{r\bar b} -\hat g^{\bar s r}g^{\bar lk}\hat R_{k\bar lr}^mg_{m\bar s}\,.
$$
\end{lemma}
\begin{proof} We directly compute 
$$
\begin{aligned}
\Box  {\rm tr}_{\hat g} g=&\hat g^{\bar sr}\dot g_{r \bar s}-g^{\bar lk} \hat \nabla_{k}\hat \nabla_{\bar l}(\hat g^{\bar s r} 	 g_{r \bar s})\\
=&\hat g^{\bar s r}\dot g_{r \bar s}-\hat g^{\bar s r}g^{\bar lk}\hat \nabla_{\bar l} \hat \nabla_{k} 	 g_{r \bar s}\\
=&\hat g^{\bar s r}(\dot g_{r \bar s}+\tilde R_{r \bar s})-\hat g^{\bar s r}(\tilde R_{r \bar s}+g^{\bar lk}\hat \nabla_{\bar l} \hat \nabla_{k} 	 g_{r \bar s}))\\
=&\hat g^{\bar s r}(\dot g_{r \bar s}+\tilde R_{r \bar s})-\hat g^{\bar s r}(-g^{\bar lk}g_{r \bar s,k\bar l}+g^{\bar lk}g^{\bar b a}g_{a\bar s,\bar l}g_{r	\bar b,k}+g^{\bar lk}\hat \nabla_{\bar l} \hat \nabla_{k} 	 g_{s \bar r})\,.
\end{aligned}
$$
Since 
$$
g_{a\bar s,\bar l}=\hat \nabla_{\bar l}g_{a\bar s}+\hat \Gamma_{\bar{l}\bar{s}}^{\bar p}g_{a\bar p}\,,\quad 
g_{r	\bar b,k}=\hat \nabla_{k}g_{r\bar b}+\hat \Gamma_{kr}^{m}g_{m\bar b} 
$$
and 
$$
\begin{aligned}
g_{r \bar s,k\bar l}=&\partial_{\bar l}(\hat{\nabla}_kg_{r \bar s}+\hat \Gamma_{kr}^mg_{m\bar s})
=\partial_{\bar l}\hat{\nabla}_kg_{r \bar s}+\partial_{\bar l}(\hat \Gamma_{kr}^m)g_{m\bar s}+
\hat \Gamma_{kr}^mg_{m\bar s,\bar l}\\
=&\hat{\nabla}_{\bar l}\hat{\nabla}_kg_{r \bar s}+\Gamma_{\bar l\bar s}^{\bar m}\hat{\nabla}_kg_{r \bar s}-\hat R_{k\bar lr}^mg_{m\bar s}+
\hat \Gamma_{kr}^m\hat{\nabla}_{\bar l}g_{m\bar s}+\hat \Gamma_{kr}^m\hat \Gamma_{\bar l\bar s}^{\bar p}g_{m\bar p}\,,
\end{aligned}
$$
it follows 
$$
-g^{\bar lk}g_{r \bar s,k\bar l}=-g^{\bar lk}\hat{\nabla}_{\bar l}\hat{\nabla}_kg_{r \bar s}
-g^{\bar lk}\Gamma_{\bar l\bar s}^{\bar m}\hat{\nabla}_kg_{r \bar m}+g^{\bar lk}\hat R_{k\bar lr}^mg_{m\bar s}-g^{\bar lk}
\hat \Gamma_{kr}^m\hat{\nabla}_{\bar l}g_{m\bar s}-g^{\bar lk}\hat \Gamma_{kr}^m\hat \Gamma_{\bar l\bar s}^{\bar p}g_{m\bar p}
$$
and 
$$
g^{\bar lk}g^{\bar b a}g_{a\bar s,\bar l}g_{r	\bar b,k}
=g^{\bar lk}g^{\bar b a}\hat \nabla_{\bar l}g_{a\bar s}\hat \nabla_{k}g_{r\bar b}+
g^{\bar lk}\hat \Gamma_{\bar{l}\bar{s}}^{\bar b}\hat \Gamma_{kr}^{m}g_{m\bar b}
+g^{\bar lk}\hat \nabla_{\bar l}g_{m\bar s}\hat \Gamma_{kr}^{m 	}
+g^{\bar lk}\hat \Gamma_{\bar{l}\bar{s}}^{\bar m}\hat \nabla_{k}g_{r\bar m}\,.
$$
Then the claim follows. 
\end{proof}
 
\begin{lemma}\label{lemma2}
We have 
$$
\Box  {\rm tr}_{g} \hat g=- g^{\bar l \alpha} (\dot g_{\alpha\bar \beta} +\tilde R_{\alpha\bar \beta}) g^{\bar \beta k} \hat g_{k \bar l}
+g^{\bar s r}g^{\bar l k} \hat{R}_{r\bar sk \bar l}-\hat g_{k\bar l}g^{\bar b a}g^{\bar r s}\Psi_{as}^k\Psi_{\bar b\bar r}^{\bar l}\,. 
$$
\end{lemma}
\begin{proof} We directly compute 
$$
\begin{aligned}
\Box  {\rm tr}_{g} \hat g=\,&\Box  (g^{\bar l k} \hat g_{k \bar l})\\
=&\,- g^{\bar l \alpha} \dot g_{\alpha\bar \beta} g^{\bar \beta k} \hat g_{k \bar l}+g^{\bar r s}\partial_{\bar r}(g^{\bar l \alpha}  g_{\alpha\bar \beta,s} g^{\bar \beta  k}\hat g_{k \bar l}
-g^{\bar l k} \hat g_{k \bar l,s})\\
=&- g^{\bar l \alpha} \dot g_{\alpha\bar \beta} g^{\bar \beta k} \hat g_{k \bar l}-g^{\bar r s}g^{\bar l \gamma}g_{\gamma\bar \delta,\bar r}g^{\bar \delta \alpha}  g_{\alpha\bar \beta,s} g^{\bar \beta k} \hat g_{k \bar l}
+g^{\bar r s}g^{\bar l \alpha}  g_{\alpha\bar \beta,s\bar r} g^{\bar \beta  k}\hat g_{k \bar l}\\
&-g^{\bar r s}g^{\bar l \alpha}  g_{\alpha\bar \beta,s} g^{\bar \beta \gamma}g_{\gamma \bar \delta,\bar r}g^{\bar 	\delta k} \hat g_{k \bar l}+g^{\bar r s}g^{\bar l \alpha}  g_{\alpha\bar \beta,s} g^{\bar \beta  k}\hat g_{k \bar l,\bar r}\\
&+g^{\bar r s}g^{\bar l a}g_{a\bar b,\bar r}g^{\bar bk} \hat g_{k \bar l,s}-g^{\bar s r}g^{\bar l k} \hat g_{k \bar l,r\bar s} \\
=&- g^{\bar l \alpha} (\dot g_{\alpha\bar \beta} +\tilde R_{\alpha\bar \beta}) g^{\bar \beta k} \hat g_{k \bar l}-g^{\bar r s}g^{\bar l \gamma}g_{\gamma\bar \delta,\bar r}g^{\bar \delta \alpha}  g_{\alpha\bar \beta,s} g^{\bar \beta k} \hat g_{k \bar l}
+g^{\bar r s}g^{\bar l \alpha}  g_{\alpha\bar \beta,s\bar r} g^{\bar \beta  k}\hat g_{k \bar l}\\
&-g^{\bar r s}g^{\bar l \alpha}  g_{\alpha\bar \beta,s} g^{\bar \beta \gamma}g_{\gamma \bar \delta,\bar r}g^{\bar 	\delta k} \hat g_{k \bar l}+g^{\bar r s}g^{\bar l \alpha}  g_{\alpha\bar \beta,s} g^{\bar \beta  k}\hat g_{k \bar l,\bar r}\\
&+g^{\bar r s}g^{\bar l a}g_{a\bar b,\bar r}g^{\bar bk} \hat g_{k \bar l,s}-g^{\bar s r}g^{\bar l k} \hat g_{k \bar l,r\bar s}
+g^{\bar l \alpha}\tilde R_{\alpha\bar \beta}g^{\bar \beta k} \hat g_{k \bar l} \\
=&- g^{\bar l \alpha} (\dot g_{\alpha\bar \beta} +\tilde R_{\alpha\bar \beta}) g^{\bar \beta k} \hat g_{k \bar l}-g^{\bar r s}g^{\bar l \gamma}g_{\gamma\bar \delta,\bar r}g^{\bar \delta \alpha}  g_{\alpha\bar \beta,s} g^{\bar \beta k} \hat g_{k \bar l}
+g^{\bar r s}g^{\bar l \alpha}  g_{\alpha\bar \beta,s\bar r} g^{\bar \beta  k}\hat g_{k \bar l}\\
&-g^{\bar r s}g^{\bar l \alpha}  g_{\alpha\bar \beta,s} g^{\bar \beta \gamma}g_{\gamma \bar \delta,\bar r}g^{\bar 	\delta k} \hat g_{k \bar l}+g^{\bar r s}g^{\bar l \alpha}  g_{\alpha\bar \beta,s} g^{\bar \beta  k}\hat g_{k \bar l,\bar r}\\
&+g^{\bar r s}g^{\bar l a}g_{a\bar b,\bar r}g^{\bar bk} \hat g_{k \bar l,s}-g^{\bar s r}g^{\bar l k} \hat g_{k \bar l,r\bar s}
-g^{\bar l \alpha} g^{\bar r s}g_{\alpha\bar 	\beta,s	\bar r} g^{\bar \beta k}\hat g_{k \bar l}
+g^{\bar r s}g^{\bar l \alpha}  g_{\alpha\bar \beta,s} g^{\bar \beta \gamma}g_{\gamma \bar \delta,\bar r}g^{\bar 	\delta k} 
 \\
=&- g^{\bar l \alpha} (\dot g_{\alpha\bar \beta} +\tilde R_{\alpha\bar \beta}) g^{\bar \beta k} \hat g_{k \bar l}-g^{\bar r s}g^{\bar l \gamma}g_{\gamma\bar \delta,\bar r}g^{\bar \delta \alpha}  g_{\alpha\bar \beta,s} g^{\bar \beta k} \hat g_{k \bar l}
+g^{\bar r s}g^{\bar l \alpha}  g_{\alpha\bar \beta,s} g^{\bar \beta  k}\hat g_{k \bar l,\bar r}\\
&+g^{\bar rs}g^{\bar l a}g_{a\bar b,\bar r}g^{\bar bk} \hat g_{k \bar l, s}-g^{\bar s r}g^{\bar l k} \hat g_{k \bar l,r\bar s}\,,
\end{aligned}
$$
i.e. 
$$
\begin{aligned}
\Box  {\rm tr}_{g} \hat g=&- g^{\bar l \alpha} (\dot g_{\alpha\bar \beta} +\tilde R_{\alpha\bar \beta}) g^{\bar \beta k} \hat g_{k \bar l}-g^{\bar r s}g^{\bar la}g_{a\bar b,\bar r}g^{\bar b \alpha}  g_{\alpha\bar \beta,s} g^{\bar \beta k} \hat g_{k \bar l}
+g^{\bar r s}g^{\bar la}  g_{a\bar b,s} g^{\bar b  k}\hat g_{k \bar l,\bar r}\\
&+g^{\bar rs}g^{\bar l a}g_{a\bar b,\bar r}g^{\bar bk} \hat g_{k \bar l, s}
-g^{\bar sr}g^{\bar l k}g^{\bar \beta \gamma}\hat g_{\gamma \bar l,s}\hat g_{k\bar \beta,r}+g^{\bar s r}g^{\bar l k} \hat R_{r\bar sk \bar l}\,.
\end{aligned}
$$
Finally we observe that 
\begin{multline*}
\hat g_{k\bar l}g^{\bar b a}g^{\bar r s}\Psi_{as}^k\Psi_{\bar b\bar r}^{\bar l}= g^{\bar r s}g^{\bar la}g_{a\bar b,\bar r}g^{\bar b \alpha}  g_{\alpha\bar \beta,s} g^{\bar \beta k} \hat g_{k \bar l}
-g^{\bar r s}g^{\bar la}  g_{a\bar b,s} g^{\bar b  k}\hat g_{k \bar l,\bar r}\\
-g^{\bar rs}g^{\bar l a}g_{a\bar b,\bar r}g^{\bar bk} \hat g_{k \bar l, s}
+g^{\bar sr}g^{\bar l k}g^{\bar \beta \gamma}\hat g_{\gamma \bar l,s}\hat g_{k\bar \beta,r}\,. 
\end{multline*}
Indeed 
$$
\begin{aligned}
&g^{\bar r s}g^{\bar la}g_{a\bar b,\bar r}g^{\bar b \alpha}  g_{\alpha\bar \beta,s} g^{\bar \beta k} \hat g_{k \bar l}
-g^{\bar r s}g^{\bar la}  g_{a\bar b,s} g^{\bar b  k}\hat g_{k \bar l,\bar r}
-g^{\bar rs}g^{\bar l a}g_{a\bar b,\bar r}g^{\bar bk} \hat g_{k \bar l, s}
+g^{\bar sr}g^{\bar l k}g^{\bar \beta \gamma}\hat g_{\gamma \bar l,s}\hat g_{k\bar \beta,r} =\\
& g^{\bar r s}\Gamma^{\bar l}_{\bar r\bar b} 
\Gamma_{s\alpha}^{k}g^{\bar b \alpha}   \hat g_{k \bar l}
-g^{\bar r s}g^{\bar la}\Gamma^{k}_{sa}  \hat \Gamma^{\bar m}_{\bar r \bar l}\hat g_{a\bar m}
-g^{\bar rs}  g^{\bar b k}
\Gamma_{\bar r\bar b}^{\bar l}\hat\Gamma_{sk}^{m} \hat g_{ m \bar l}
+g^{\bar rs}  g^{\bar b k}
\hat \Gamma_{\bar r\bar b}^{\bar l}\hat\Gamma_{sk}^{m} \hat g_{ m \bar l}
\end{aligned}
$$
and the claim follows.
\end{proof}

\begin{lemma}\label{lemma3}
We have 
$$
\Box  \log\frac{\det g}{\det \hat g}={\rm tr}_g(\dot g+\widetilde {\rm Ric})-{\rm tr}_{g}(\hat{\rm Ric})\,.
$$
\end{lemma}
\begin{proof}
We directly compute 
$$
\begin{aligned}
\Box \log\frac{\det g}{\det \hat g}=&g^{\bar sr}\dot{g}_{r\bar s}-g^{\bar s r}\partial_{r}\partial_{\bar s}(\log \det g)+g^{\bar s r}\partial_{r}\partial_{\bar s}(\log \det \hat g)\\
=&{\rm tr}_g(\dot g+\widetilde {\rm Ric})-{\rm tr}_{g}(\hat{\rm Ric})\,,
\end{aligned}
$$
as required. 
\end{proof}

Next we compute $\square |T|_g^2$ and $\square |{\rm Rm}|_g^2$. 

\begin{lemma}\label{boxT} We have
\begin{equation*}
\begin{split}
\square\,|T|^2_g=&-|\nabla T|^2_g-|\bar\nabla T|^2_g+4\operatorname{Re}\left(g^{\bar pi}g^{\bar qj}\nabla_i(\dot g_{j\bar r}+\tilde R_{j\bar r})T_{\bar p\bar q}^{\bar r}\right)\\&-2g^{\bar bi}g^{\bar qj}g^{\bar pa}g_{k\bar r}T_{ij}^k(\dot g_{a\bar b}+\tilde R_{a\bar b})T_{\bar p\bar q}^{\bar r}+ g^{\bar pi}g^{\bar qj}T_{ij}^k(\dot g_{k\bar r}+\tilde R_{k\bar r})T_{\bar p\bar q}^{\bar r}\\&+
4\operatorname{Re}\left(g^{\bar pi}g^{\bar qj}g^{\bar ba}g_{k\bar r}T_{aj}^mT_{\bar p\bar q}^{\bar r}R_{m\bar bi}^k\right)\,.
\end{split}
\end{equation*}
\end{lemma}
\begin{proof}
    Since $|T|^2_g=g^{\bar ji} g^{\bar nm}g_{a\bar b}T_{im}^aT_{\bar j\bar n}^{\bar b}$, we have
    \begin{equation*}\begin{split}
        \partial_t|T|^2_g=&-g^{\bar js}g^{\bar ri}\dot g_{s\bar r}g^{\bar nm}g_{a\bar b}T_{im}^aT_{\bar j\bar n}^{\bar b}-g^{\bar ji}g^{\bar ns}g^{\bar rm}\dot g_{s\bar r}g_{a\bar b}T_{im}^aT_{\bar j\bar n}^{\bar b}+g^{\bar ji} g^{\bar nm}\dot g_{a\bar b}T_{im}^aT_{\bar j\bar n}^{\bar b}\\&+g^{\bar ji} g^{\bar nm}g_{a\bar b}\dot T_{im}^aT_{\bar j\bar n}^{\bar b}+g^{\bar ji} g^{\bar nm}g_{a\bar b}T_{im}^a\dot T_{\bar j\bar n}^{\bar b}\,.
    \end{split}\end{equation*}
Furthermore\begin{equation*}
    \dot T_{im}^a=\dot \Gamma_{im}^a-\dot \Gamma_{mi}^a=\nabla_i\dot g_m^a-\nabla_m\dot g_i^a\,.
\end{equation*}
Now, since
\begin{equation*}
    \nabla_{\bar b}T_{ij}^k=-R_{i\bar bj}^k+R_{j\bar bi}^k\,,
\end{equation*} we have
\begin{equation*}\begin{split}
    g^{\bar ba}\nabla_a\nabla_{\bar b}T_{ij}^k=&g^{\bar ba}\nabla_a(-R_{i\bar bj}^k+R_{j\bar bi}^k)=g^{\bar ba}(-\nabla_iR_{a\bar bj}^k+T_{ai}^mR_{m\bar bj}^k+\nabla_jR_{a\bar bi}^k-T_{aj}^mR_{m\bar bi}^k)\\=&-\nabla_i\tilde R_j^k+\nabla_j\tilde R_i^k+g^{\bar ba}T_{ai}^mR_{m\bar bj}^k-g^{\bar ba}T_{aj}^mR_{m\bar bi}^k\,.\end{split}
\end{equation*}
Similarly, we have
\begin{equation*}
    g^{\bar ba}\nabla_{\bar b}\nabla_aT_{\bar p\bar q}^{\bar r}=-\nabla_{\bar p}\tilde R_{\bar q}^{\bar r}+\nabla_{\bar q}\tilde R_{\bar p}^{\bar r}+g^{\bar ba}T_{\bar b\bar p}^{\bar n}R_{\bar na\bar q}^{\bar r}-g^{\bar ba}T_{\bar b\bar q}^{\bar n}R_{\bar na\bar p}^{\bar r}\,.
\end{equation*}
Then,
\begin{equation*}
    \begin{split}
        \Delta \,|T|^2_g=|\nabla T|^2_g+|\bar\nabla T|^2_g+g^{\bar pi}g^{\bar qj}g_{k\bar r}g^{\bar ba}(\nabla_ a\nabla_{\bar b}T_{ij}^kT_{\bar p\bar q}^{\bar r}+T_{ij}^k\nabla_a\nabla_{\bar b}T_{\bar p\bar q}^{\bar r})
    \end{split}
\end{equation*} and\begin{equation*}
    [\nabla_a,\nabla_{\bar b}]T_{\bar p\bar q}^{\bar r}=R_{\bar ba\bar p}^{\bar n}T_{\bar n\bar q}^{\bar r}+R_{\bar ba\bar q}^{\bar n}T_{\bar p\bar n}^{\bar r}-R_{\bar ba\bar n}^{\bar r}T_{\bar p\bar q}^{\bar n}\,,
\end{equation*}hence\begin{equation*}\begin{split}
\Delta\,|T|^2_g=&\,|\nabla T|^2_g+|\bar\nabla T|^2_g+g^{\bar pi}g^{\bar qj}g_{k\bar r}g^{\bar ba}(\nabla_ a\nabla_{\bar b}T_{ij}^kT_{\bar p\bar q}^{\bar r}+T_{ij}^k(\nabla_{\bar b}\nabla_aT_{\bar p\bar q}^{\bar r}+R_{\bar ba\bar p}^{\bar n}T_{\bar n\bar q}^{\bar r}+R_{\bar ba\bar q}^{\bar n}T_{\bar p\bar n}^{\bar r}-R_{\bar ba\bar n}^{\bar r}T_{\bar p\bar q}^{\bar n}))\\
=&\,|\nabla T|^2_g+|\bar\nabla T|^2_g+g^{\bar pi}g^{\bar qj}g_{k\bar r}(-\nabla_i\tilde R_j^k+\nabla_j\tilde R_i^k+g^{\bar ba}T_{ai}^mR_{m\bar bj}^k-g^{\bar ba}T_{aj}^mR_{m\bar bi}^k)T_{\bar p\bar q}^{\bar r}\\
+&g^{\bar pi}g^{\bar qj}g_{k\bar r}T_{ij}^k(-\nabla_{\bar p}\tilde R_{\bar q}^{\bar r}+\nabla_{\bar q}\tilde R_{\bar p}^{\bar r}+g^{\bar ba}T_{\bar b\bar p}^{\bar n}R_{\bar na\bar q}^{\bar r}-g^{\bar ba}T_{\bar b\bar q}^{\bar n}R_{\bar na\bar p}^{\bar r})\\+&g^{\bar pi}g^{\bar qj}g_{k\bar r}g^{\bar ba}T_{ij}^k(R_{\bar ba\bar p}^{\bar n}T_{\bar n\bar q}^{\bar r}+R_{\bar ba\bar q}^{\bar n}T_{\bar p\bar n}^{\bar r}-R_{\bar ba\bar n}^{\bar r}T_{\bar p\bar q}^{\bar n})\,.
\end{split}\end{equation*}
It follows 
\begin{equation*}\begin{split}
    \square\,|T|^2_g=&-|\nabla T|^2_g-|\bar\nabla T|^2_g+g^{\bar pi}g^{\bar qj}g_{k\bar r}\nabla_i(\dot g_j^k+\tilde R_j^k)T_{\bar p\bar q}^{\bar r}-g^{\bar pi}g^{\bar qj}g_{k\bar r}\nabla_j(\dot g_i^k+\tilde R_i^k)T_{\bar p\bar q}^{\bar r}\\&+g^{\bar pi}g^{\bar qj}g_{k\bar r}T_{ij}^k\nabla_{\bar p}(\dot g_{\bar q}^{\bar r}+\tilde R_{\bar q}^{\bar r})-g^{\bar pi}g^{\bar qj}g_{k\bar r}T_{ij}^k\nabla_{\bar q}(\dot g_{\bar p}^{\bar r}+\tilde R_{\bar p}^{\bar r})-g^{\bar bi}g^{\bar qj}g_{k\bar r}T_{ij}^k(\dot g_{\bar b}^{\bar p}+\tilde R_{\bar b}^{\bar p})T_{\bar p\bar q}^{\bar r}\\&-g^{\bar pi}g^{\bar bj}T_{ij}^k(\dot g_{\bar b}^{\bar q}+\tilde R_{\bar b}^{\bar p})T_{\bar p\bar q}^{\bar r}+ g^{\bar pi}g^{\bar qj}T_{ij}^k(\dot g_{k\bar r}+\tilde R_{k\bar r})T_{\bar p\bar q}^{\bar r}\\&+g^{\bar pi}g^{\bar qj}g^{\bar ba}g_{k\bar r}(T_{aj}^mT_{\bar p\bar q}^{\bar r}R_{m\bar bi}^k-T_{ai}^mT_{\bar p\bar q}^{\bar r}R_{m\bar bj}^k+T_{ij}^kT_{\bar b\bar q}^{\bar n}R_{\bar na\bar p}^{\bar r}-T_{ij}^kT_{\bar b\bar p}^{\bar n}R_{\bar na\bar q}^{\bar r})\,.
\end{split}\end{equation*}
Finally,
\begin{equation*}
\begin{split}
\square\,|T|^2_g=&-|\nabla T|^2_g-|\bar\nabla T|^2_g\\
&+2g^{\bar pi}g^{\bar qj}T_{\bar p\bar q}^{\bar r}\nabla_i(\dot g_{j\bar r}+\tilde R_{j\bar r})+2g^{\bar pi}g^{\bar qj}T_{ij}^k\nabla_{\bar p}(\dot g_{k\bar q}+\tilde R_{k\bar q})\\
&-g^{\bar bi}g^{\bar qj}g_{k\bar r}T_{ij}^kT_{\bar p\bar q}^{\bar r}(\dot g_{\bar b}^{\bar p}+\tilde R_{\bar b}^{\bar p})\\
&-g^{\bar pi}g^{\bar bj}T_{ij}^kT_{\bar p\bar q}^{\bar r}(\dot g_{\bar b}^{\bar q}+\tilde R_{\bar b}^{\bar p})+ g^{\bar pi}g^{\bar qj}T_{ij}^kT_{\bar p\bar q}^{\bar r}(\dot g_{k\bar r}+\tilde R_{k\bar r})\\
&+2g^{\bar pi}g^{\bar qj}g^{\bar ba}g_{k\bar r}(T_{aj}^mT_{\bar p\bar q}^{\bar r}R_{m\bar bi}^k+T_{ij}^kT_{\bar b\bar q}^{\bar n}R_{\bar na\bar p}^{\bar r})\,,
\end{split}
\end{equation*}
and the claim follows. 
\end{proof}

\begin{lemma}\label{evcurv} We have
\begin{equation*}
    \begin{split}
        \square\,|\operatorname{Rm}|_g^2 =&-|\nabla\operatorname{Rm}|_g^2-|\bar\nabla\operatorname{Rm}|_g^2-2\operatorname{Re}\left[g^{\bar ni}g^{\bar jm}g^{\bar pk}\nabla_{\bar j}\nabla_i(\dot g_{k\bar q}+\tilde R_{k\bar q})R_{\bar nm\bar p}^{\bar q}\right] \\&-g^{\bar ni}g^{\bar jm}g^{\bar pk}g_{l\bar q}R_{i\bar jk}^l(\dot g_{\bar n}^{\bar r}+\tilde R_{\bar r}^{\bar n})R_{\bar rm\bar p}^{\bar  q}-g^{\bar ni}g^{\bar jm}g^{\bar pk}g_{l\bar q}R_{i\bar jk}^l(\dot g_{\bar p}^{\bar r}+\tilde R_{\bar p}^{\bar q})R_{\bar nm\bar r}^{\bar q}\\&+g^{\bar ni}g^{\bar jm}g^{\bar pk}R_{i\bar jk}^l(\dot g_{l\bar r}+\tilde R_{l\bar r})R_{\bar nm\bar p}^{\bar r}-g^{\bar ni}g^{\bar ja}g^{\bar pk}g_{l\bar q}R_{i\bar jk}^l(\dot g_a^m+\tilde R_a^m)R_{\bar nm\bar p}^q\\&+2\operatorname{Im}\left[g^{\bar ni}g^{\bar jm}g^{\bar pk}g_{l\bar q}^{\bar ba}T_{ai}^s\nabla_{\bar j}R_{s\bar bk}^lR_{\bar nm\bar p}^{\bar q}\right]-2\operatorname{Im}\left[g^{\bar ni}g^{\bar jm}g^{\bar pk}g_{l\bar q}^{\bar ba}R_{i\bar jk}^lT_{am}^s\nabla_{\bar b}R_{\bar ns\bar p}^{\bar q}\right]\\&+2\operatorname{Im}\left[g^{\bar ni}g^{\bar jm}g^{\bar pk}g_{l\bar q}g^{\bar ba}(-R_{a\bar js}^lR_{i\bar bk}^s+R_{i\bar ja}^sR_{s\bar bk}^l+R_{a\bar jk}^sR_{i\bar bs}^l)R_{\bar nm\bar p}^{\bar q}\right]\,.
    \end{split}
\end{equation*}
\end{lemma}
\begin{proof}
 We directly compute
\begin{equation*}\begin{split}
\partial_t|{\rm Rm}|_g^2=&\partial_t\left[g^{\bar ni}g^{\bar jm}g^{\bar pk}g_{l\bar q}R_{i\bar jk}^lR_{\bar nm\bar p}^{\bar q}\right]\\=&-g^{\bar na}g^{\bar bi}\dot g_{a\bar b}g^{\bar jm}g^{\bar pk}g_{l\bar q}R_{i\bar jk}^lR_{\bar nm\bar p}^{\bar q}-g^{\bar ni}g^{\bar ja}g^{\bar bm}\dot g_{a\bar b}g^{\bar pk}g_{l\bar k}R_{i\bar jk}^lR_{\bar nm\bar p}^{\bar q}\\
&-g^{\bar ni}g^{\bar jm}g^{\bar pa}g^{\bar bk}\dot g_{a\bar b}g_{l\bar k}R_{i\bar jk}^lR_{\bar nm\bar p}^{\bar q}+g^{\bar ni}g^{\bar jm}g^{\bar pk}\dot g_{l\bar q}R_{i\bar jk}^lR_{\bar nm\bar p}^{\bar q}\\
&-g^{\bar ni}g^{\bar jm}g^{\bar pk}g_{l\bar q}\nabla_{\bar j}\nabla_i\dot g_{k}^lR_{\bar nm\bar p}^{\bar q}-g^{\bar ni}g^{\bar jm}g^{\bar pk}g_{l\bar q}R_{i\bar jk}^l\nabla_m\nabla_{\bar n}\dot g_{\bar p}^{\bar q}   \,.
\end{split}
\end{equation*}
Furthermore,
\begin{equation*}\begin{split}
[\nabla_a,\nabla_{\bar b}]R_{\bar nm\bar p}^{\bar q}=R_{\bar ba\bar n}^{\bar r}R_{\bar rm\bar p}^{\bar q}+R_{\bar ba\bar p}^{\bar r}R_{\bar nm\bar r}^{\bar q}-R_{\bar ba\bar r}^{\bar q}R_{\bar nm\bar p}^{\bar r}-R_{a\bar bm}^sR_{\bar ns\bar p}^{\bar q}\,,   
\end{split}
\end{equation*}
and
\begin{equation*}
    [\nabla_a,\nabla_{\bar j}]R_{i\bar bk}^l=R_{a\bar js}^lR_{i\bar bk}^s+R_{\bar ja\bar b}^{\bar r}R_{i\bar rk}^l-R_{a\bar ji}^sR_{s\bar bk}^l-R_{a\bar jk}^sR_{i\bar bs}^l\,,
\end{equation*}hence
\begin{equation*}\begin{split}
g^{\bar ba}\nabla_a\nabla_{\bar b}R_{i\bar jk}^l=
&\nabla_{\bar j}\nabla_i\tilde R_k^l-g^{\bar ba}T_{ai}^s\nabla_{\bar j}R_{s\bar bk}^l -g^{\bar ba}(\nabla_aR_{i\bar rk}^l)T_{\bar b\bar j}^{\bar r}\\
+&g^{\bar ba}(R_{a\bar js}^lR_{i\bar bk}^s+R_{\bar ba\bar j}^{\bar r}R_{i\bar rk}^l-R_{i\bar ja}^sR_{s\bar bk}^l-R_{a\bar jk}^sR_{i\bar bs}^l)\,.
\end{split}\end{equation*}and\begin{equation*}
\begin{split}
g^{\bar ba}\nabla_{\bar b}\nabla_aR_{\bar n m\bar p}^{\bar q}&= \nabla_m\nabla_{\bar n}\tilde R_{\bar p}^{\bar q}-g^{\bar ba}\nabla_mR_{\bar ra\bar p}T_{\bar b\bar n}^{\bar r}-g^{\bar ba}T_{am}^s\nabla_{\bar b}R_{\bar ns\bar p}^{\bar q} \\
&+g^{\bar ba}(R_{a\bar bm}^sR_{\bar ns\bar p}^{\bar q}-R_{\bar nm\bar b}R_{\bar ra\bar p}^{\bar q}+R_{\bar bm\bar r}^{\bar q}R_{\bar na\bar p}^{\bar r}-R_{\bar bm\bar p}^{\bar r}R_{\bar na\bar r}^{\bar q})\,.
\end{split}
\end{equation*}
Then,
\begin{equation*}\begin{split}
\Delta\,|{\rm Rm}|_g^2=
&\,|\nabla\operatorname{Rm}|^2+|\bar\nabla\operatorname{Rm}|^2+g^{\bar ni}g^{\bar jm}g^{\bar pk}g_{l\bar q}\left(g^{\bar ba}\nabla_a\nabla_{\bar b}R_{i\bar jk}^lR_{\bar nm\bar p}^{\bar q}+g^{\bar ba}R_{i\bar jk}^l\nabla_a\nabla_{\bar b}R_{\bar nm\bar p}^{\bar q}\right) \\ 
=&\,|\nabla\operatorname{Rm}|_g^2+|\bar\nabla\operatorname{Rm}|_g^2+g^{\bar ni}g^{\bar jm}g^{\bar pk}g_{l\bar q}\left(g^{\bar ba}\nabla_a\nabla_{\bar b}R_{i\bar jk}^lR_{\bar nm\bar p}^{\bar q}+g^{\bar ba}R_{i\bar jk}^l\nabla_{\bar b}\nabla_aR_{\bar nm\bar p}^{\bar q}\right)\\
&+g^{\bar ni}g^{\bar jm}g^{\bar pk}g_{l\bar q}g^{\bar ba}R_{i\bar jk}^l\left(R_{\bar ba\bar n}^{\bar r}R_{\bar rm\bar p}^{\bar q}+R_{\bar ba\bar p}^{\bar r}R_{\bar nm\bar r}^{\bar q}-R_{\bar ba\bar r}^{\bar q}R_{\bar nm\bar p}^{\bar r}-R_{a\bar bm}^sR_{\bar ns\bar p}^{\bar q}\right) \\
=&\,|\nabla\operatorname{Rm}|_g^2+|\bar\nabla\operatorname{Rm}|_g^2+g^{\bar ni}g^{\bar jm}g^{\bar pk}g_{l\bar q}\nabla_{\bar j}\nabla_i\tilde R_k^lR_{\bar nm\bar p}^{\bar q}+g^{\bar ni}g^{\bar jm}g^{\bar pk}g_{l\bar q}R_{i\bar jk}^l\nabla_m\nabla_{\bar n}\tilde R_{\bar p}^{\bar q}\\
&+g^{\bar ni}g^{\bar jm}g^{\bar pk}g_{l\bar q}\left[-g^{\bar ba}T_{ai}^s\nabla_{\bar j}R_{s\bar bk}^l -g^{\bar ba}(\nabla_aR_{i\bar rk}^l)T_{\bar b\bar j}^{\bar r}+g^{\bar ba}(R_{a\bar js}^lR_{i\bar bk}^s+R_{\bar ba\bar j}^{\bar r}R_{i\bar rk}^l\right.\\
&-\left.R_{i\bar ja}^sR_{s\bar bk}^l-R_{a\bar jk}^sR_{i\bar bs}^l)\right]R_{\bar nm\bar p}^{\bar q}-g^{\bar ni}g^{\bar jm}g^{\bar pk}g_{l\bar q}R_{i\bar jk}^l\left[-g^{\bar ba}\nabla_mR_{\bar ra\bar p}T_{\bar b\bar n}^{\bar r}\right.\\
&\left.-g^{\bar ba}T_{am}^s\nabla_{\bar b}R_{\bar ns\bar p}^{\bar q} +g^{\bar ba}(R_{a\bar bm}^sR_{\bar ns\bar p}^{\bar q}-R_{\bar nm\bar b}R_{\bar ra\bar p}^{\bar q}+R_{\bar bm\bar r}^{\bar q}R_{\bar na\bar p}^{\bar r}-R_{\bar bm\bar p}^{\bar r}R_{\bar na\bar r}^{\bar q})\right]\\
&+g^{\bar ni}g^{\bar jm}g^{\bar pk}g_{l\bar q}g^{\bar ba}R_{i\bar jk}^l\left(R_{\bar ba\bar n}^{\bar r}R_{\bar rm\bar p}^{\bar q}+R_{\bar ba\bar p}^{\bar r}R_{\bar nm\bar r}^{\bar q}-R_{\bar ba\bar r}^{\bar q}R_{\bar nm\bar p}^{\bar r}-R_{a\bar bm}^sR_{\bar ns\bar p}^{\bar q}\right)\,.
\end{split}\end{equation*}
It follows 
\begin{equation*}
\begin{split}
\square\,|\operatorname{Rm}|_g^2 =
&-|\nabla\operatorname{Rm}|_g^2-|\bar\nabla\operatorname{Rm}|_g^2-2\operatorname{Re}\left[g^{\bar ni}g^{\bar jm}g^{\bar pk}\nabla_{\bar j}\nabla_i(\dot g_{k\bar q}+\tilde R_{k\bar q})R_{\bar nm\bar p}^{\bar q}\right] \\
&-g^{\bar ni}g^{\bar jm}g^{\bar pk}g_{l\bar q}R_{i\bar jk}^l(\dot g_{\bar n}^{\bar r}+\tilde R_{\bar r}^{\bar n})R_{\bar rm\bar p}^{\bar  q}-g^{\bar ni}g^{\bar jm}g^{\bar pk}g_{l\bar q}R_{i\bar jk}^l(\dot g_{\bar p}^{\bar r}+\tilde R_{\bar p}^{\bar q})R_{\bar nm\bar r}^{\bar q}\\
&+g^{\bar ni}g^{\bar jm}g^{\bar pk}R_{i\bar jk}^l(\dot g_{l\bar r}+\tilde R_{l\bar r})R_{\bar nm\bar p}^{\bar r}-g^{\bar ni}g^{\bar ja}g^{\bar pk}g_{l\bar q}R_{i\bar jk}^l(\dot g_a^m+\tilde R_a^m)R_{\bar nm\bar p}^q\\
&-g^{\bar ni}g^{\bar jm}g^{\bar pk}g_{l\bar q}\left[-g^{\bar ba}T_{ai}^s\nabla_{\bar j}R_{s\bar bk}^l -g^{\bar ba}(\nabla_aR_{i\bar rk}^l)T_{\bar b\bar j}^{\bar r}+g^{\bar ba}R_{a\bar js}^lR_{i\bar bk}^s\right.\\
-&\left.R_{i\bar ja}^sR_{s\bar bk}^l-R_{a\bar jk}^sR_{i\bar bs}^l)\right]R_{\bar nm\bar p}^{\bar q}+g^{\bar ni}g^{\bar jm}g^{\bar pk}g_{l\bar q}R_{i\bar jk}^l\left[-g^{\bar ba}\nabla_mR_{\bar ra\bar p}T_{\bar b\bar n}^{\bar r}-g^{\bar ba}T_{am}^s\nabla_{\bar b}R_{\bar ns\bar p}^{\bar q} \right.\\
&\left.+g^{\bar ba}(R_{a\bar bm}^sR_{\bar ns\bar p}^{\bar q}-R_{\bar nm\bar b}R_{\bar ra\bar p}^{\bar q}+R_{\bar bm\bar r}^{\bar q}R_{\bar na\bar p}^{\bar r}-R_{\bar bm\bar p}^{\bar r}R_{\bar na\bar r}^{\bar q})\right]\\
+&g^{\bar ni}g^{\bar jm}g^{\bar pk}g_{l\bar q}g^{\bar ba}R_{i\bar jk}^lR_{a\bar bm}^sR_{\bar ns\bar p}^{\bar q}\,,
    \end{split}
\end{equation*}
which implies the claim. 
\end{proof}


%

\section{General Estimates}
Let us consider the same setting of the previous section: a compact Hermitian manifold $(M,\hat g)$ with a smooth curve $g=g(t)$ of Hermitian metrics. We assume that there exists a positive constant $K$ such that 
$$
\frac{1}{K}\hat g\leq g(t)\leq K\hat g\,,
$$ 
for every $t$ in the time domain of $g$. We further denote by 
$C$ a positive constant which may depend only on $M$, $\hat g$, $K$ and $g(0)$. The constant $C$ will be
used repeatedly and may change from line to line.
%
\begin{lemma}
We have 
\begin{eqnarray}
 &&\label{prima} |T|_g\leq  C(S^{1/2}+1)\,,\\
&&\label{seconda} |T|^2_g\leq C(S+1)\,,\\
&&\label{terza} |\nabla T|_{g}\leq C(|\nabla \Psi|_g+S^{1/2}+1)\,,\\
&&\label{quarta} |\bar \nabla T|_{g}\leq C(|\bar \nabla \Psi|_g+S^{1/2}+1)\,,\\
&&\label{quinta}|\nabla \bar \nabla T|_g\leq C\left(\left|\nabla\bar \Psi\right|_g|T|_g+S^{1/2}\left|\hat\nabla T\right|_g+S|T|_g+\left|\hat \nabla\bar{\hat\nabla}T \right|_g\right)
\end{eqnarray}
\end{lemma}
\begin{proof}
Since 
$$
T_{ij}^k-\hat T_{ij}^k=\Psi_{ij}^k-\Psi_{ji}^k\,,
$$
we get 
$$
|T|_g\leq 2S^{1/2}+|\hat T|_g
$$
and \eqref{prima} and \eqref{seconda} follow. Finally, 
$$
|\nabla T|_g\leq |\nabla (T-\hat T)|_g+ |\nabla \hat T|_g\leq C\,|\nabla \Psi|_g+|\nabla \hat T|_g\,. 
$$
Then writing 
$$
\nabla \hat T=\Psi *\hat T+\hat\nabla \hat T\,,
$$
\eqref{terza}. 

Similarly, 
$$
|\bar \nabla T|_g\leq |\bar \nabla (T-\hat T)|_g+ |\bar \nabla \hat T|_g\leq C\,(|\bar \nabla \Psi|_g+S^{1/2}+1)\,,  
$$
and  \eqref{quarta} follows. 

Finally 
$$
\begin{aligned}
|\nabla \bar \nabla T|_g=&\left|\nabla\left(\bar \Psi* T+\bar {\hat \nabla} T\right)\right|_g=
\left|\nabla\bar\Psi* T+\bar\Psi*\nabla T+\nabla \bar {\hat \nabla} T\right|_g\\
\leq &C\left|\nabla\bar \Psi\right|_g|T|_g+\left|\bar \Psi*\hat\nabla  T+\bar \Psi*\Psi * T + \Psi* \bar{\hat \nabla} T+\hat \nabla \bar{\hat \nabla}  T)\right|_g\\
\leq &C\left(\left|\nabla\bar \Psi\right|_g|T|_g+S^{1/2}\left|\hat\nabla T\right|_g+S|T|_g+\left|\hat \nabla\bar{\hat\nabla}T \right|_g\right)
\end{aligned}
$$
and \eqref{quinta} follows. 
\end{proof}

\begin{prop}\label{lemma4.2}
We have 
 $$
\begin{aligned}
\Box S\leq & -|\nabla \Psi|_g^2-|\bar \nabla \Psi|_g^2+C|\nabla(\dot g+\tilde R)|_gS^{1/2}+C|\dot g+\tilde R|_gS\\
&+C |T|_g |\bar \nabla \Psi|_g S^{1/2}+C(1+S)\,.
\end{aligned}
$$   
\end{prop}
\begin{proof}
Since 
$$
g^{\bar mi}g^{\bar pj}g^{\bar ba}\nabla_a\hat R_{i\bar bj	\bar q}=\Psi*\hat  R+g^{\bar mi}g^{\bar pj}g^{\bar ba}\hat \nabla_a\hat R_{i\bar bj	\bar q}\,,
$$
we deduce 
$$
g^{\bar mi}g^{\bar pj}g^{\bar ba}\nabla_a\hat R_{i\bar bj	\bar q}\Psi_{\bar m\bar p}^{\bar q}\leq C(S+S^{1/2})\leq C(1+S)\,. 
$$
Moreover, by using \eqref{seconda}, we have 
$$
g^{\bar mi}g^{\bar pj}g^{\bar ba}T_{ai}^k\hat R_{k\bar bj\bar q}\Psi_{\bar m\bar p}^{\bar q}
\leq C|T|_g S^{1/2}\leq C(|T|_g^2 +S)\leq C(1+S)
$$ 
and 
$$
\begin{aligned}
g^{\bar mi}g^{\bar pj}g^{\bar ba}T_{ai}^k\nabla_{\bar b}\Psi_{kj\bar q}\Psi_{\bar m\bar p}^{\bar q}\leq &\, C |T|_g |\bar \nabla \Psi|_g S^{1/2}\,. 
\end{aligned}
$$
Hence the claim follows. 
\end{proof}

Next we observe that Lemma \ref{lemma1} implies the following  
\begin{prop}\label{prop1}
We have 
$$
\Box  {\rm tr}_{\hat g} g\leq C |\dot g+\widetilde{\rm Ric}|_g-\frac{1}{K}S +C\,.
$$
\end{prop}

In the last proposition of the present section we assume that $g$ satisfies further assumptions \eqref{ass2} and \eqref{ass3}: 

\begin{prop}\label{4.3}
Assume that $g$ satisfies  \eqref{ass2} and \eqref{ass3}, then 
\begin{equation*}
\begin{aligned}
\Box S\leq & -\frac12|\nabla \Psi|_g^2-\frac12|\bar \nabla \Psi|_g^2
+C(1+S+S^{3/2})+C |T|_g |\bar \nabla \Psi|_g S^{1/2}\,.
\end{aligned}
\end{equation*}
If further $g$ satisfies \eqref{ass4}, then 
\begin{equation}\label{boxS2}
\Box S\leq C(1+S+S^{3/2})\,.
\end{equation}
\end{prop}
\begin{proof} By combining Proposition \ref{lemma4.2} 
with assumptions  \eqref{ass1}--\eqref{ass3} we obtain 
 $$
\begin{aligned}
\Box S\leq & -|\nabla \Psi|_g^2-|\bar \nabla \Psi|_g^2+C (|\nabla \Psi|_g+ |\bar \nabla \Psi|_g)S^{1/2}+C |T|_g |\bar \nabla \Psi|_g S^{1/2}\\
&+C(1+S+S^{3/2})\,,
\end{aligned}
$$ 
and Young's inequality implies the statement. 

\medskip 
If $|T|_g\leq C$, then \eqref{boxS2} follows. 
\end{proof}

%

\section{Proof of Theorem \ref{main1}}

We can now give the proof of Theorem \ref{main1}. We will proceed by applying the maximum principle to some function $f$. We observe that the function under study can be seen as a non-local version of the one studied in \cite[Theorem 1.1]{SW2}.
\begin{proof}[Proof of Theorem $\ref{main1}$]
Let 
$$
f:=a-{\rm tr}_{\hat g}g\leq a
$$
where $a$ is a positive constant big enough so that 
$$
f\geq \frac{a}{2}\,,
$$
and let 
$$
G=\frac{S}{f}+A{\rm tr}_{\hat g}g\,,
$$
where $A$ is a positive constant which will be determined later. We have 
$$
\Box G=\frac{S}{f^2}\,\Box {\rm tr}_{\hat g}g+\frac{1}{f}\Box S-\frac{2}{f^2}\,{\rm Re}\left(\nabla {\rm tr}_{\hat g}g\cdot \bar \nabla S\right)-\frac{2S }{f^3}\,|\nabla {\rm tr}_{\hat g}g|_g+A\Box{\rm tr}_{\hat g}g
$$
and 
$$
\bar\nabla G=\frac{\bar \nabla S}{f}+\left(\frac{C S}{f^2}+A\right)\,\bar \nabla {\rm tr}_{\hat g}g\,.
$$
At a maximum point $(x_0,t_0)$ of $G$ we have 
$$
0\leq \frac{S}{f^2}\,\Box {\rm tr}_{\hat g}g+\frac{1}{f}\Box S-\frac{2}{f^2}\,{\rm Re}\left(\nabla {\rm tr}_{\hat g}g\cdot \bar \nabla S\right)-\frac{2S }{f^3}\,|\nabla {\rm tr}_{\hat g}g|_g^2+A\Box{\rm tr}_{\hat g}g
$$
and 
$$
-\frac{\bar \nabla S}{f}=\left(\frac{ S}{f^2}+A\right)\,\bar \nabla {\rm tr}_{\hat g}g\,.
$$
The last equation implies 
$$
-\frac{2}{f^2}\,{\rm Re}\left(\nabla {\rm tr}_{\hat g}g\cdot \bar \nabla S\right)
= 
\frac{2}{f}\left(\frac{ S}{f^2}+A\right)\,|\nabla {\rm tr}_{\hat g}g|^2_g
$$
at $(x_0,t_0)$ which gives 
$$
0\leq \frac{C S}{f^2}\,\Box {\rm tr}_{\hat g}g+\frac{1}{f}\Box S
+\frac{2A}{f}\,|\nabla {\rm tr}_{\hat g}g|_g^2+A\Box{\rm tr}_{\hat g}g\,,\quad \mbox{ at }(x_0,t_0)\,.
$$
By using \eqref{boxS2} and Proposition \ref{prop1}, we get 
$$
0\leq \frac{ CS}{f^2}\,\left(-\frac{S}{K}+C\right)+\frac{C}{f}(S^{3/2}+S+1)
+\frac{2ACS}{f}+A\left(-\frac{S}{K}+C\right)\,,\quad \mbox{ at }(x_0,t_0)\,.
$$
It follows 
$$
0\leq  -\frac{CS^2 }{Kf^2}+\frac{C}{f}S^{3/2}+S\left(\frac{C^2}{f^2}+\frac{C}{f}+\frac{2AC}{f}-\frac{A}{K}\right)+\frac{C}{f}+AC\,,\quad \mbox{ at }(x_0,t_0)\,.
$$
Suppose now that $S>1$ at $(x_0,t_0)$, we get 
$$
0\leq  S\left(\frac{4C^2}{a^2}+\frac{2C}{a}+\frac{4AC}{a}-\frac{A}{K}\right)+\frac{C}{f}+AC\,,\quad \mbox{ at }(x_0,t_0)\,.
$$
By choosing 
$$
a\geq 8CK\,,
$$
we get 
$$
0\leq  S\left(\frac{C}{a^2}+\frac{C}{a}-\frac{A}{2K}\right)+\frac{C}{a}+AC\,,\quad \mbox{ at }(x_0,t_0)\,.
$$

By assuming 
$$
A=2K\left(1+\frac{C}{a^2}+\frac{C}{a}\right)\,,
$$ 
we have 
$$
\frac{C}{a^2}+\frac{C}{a}-\frac{A}{2K}=-1\,,
$$
and 
$$
S\leq \frac{C}{a}+AC\,,\quad \mbox{ at }(x_0,t_0)\,,
$$
which implies a uniform bound on $S$, as required. 
\end{proof}
 
\section{Applications to Hermitian curvature flows}\label{applications}
In this section we apply Theorem \ref{main1} to the study of Hermitian curvature flows, in particular we prove Theorems \ref{ThHCF} and  \ref{ThH=0}. 

\medskip 
We consider the following preliminary lemma: 
\begin{lemma}\label{61}
Let $g=g(t)$ be a smooth curve of Hermitian metrics on a compact Hermitian manifold $(M,\hat g)$. Then:
\begin{enumerate}
\item[i.] ${\rm tr}_{g}\hat g\leq K\,,\quad  {\rm tr}_{ g}(\dot g+\widetilde{\rm Ric})\leq K \,\,\Longrightarrow\,\, {\rm tr}_{\hat g}g\leq C {\rm e}^{Ct}\,;$

\vspace{0.1cm}
\item[ii.] ${\rm tr}_{g}\hat g\leq K\,,\quad  {\rm tr}_{ g}(\dot g+\widetilde{\rm Ric})\leq 0\,,\quad  \hat {\rm Ric}\geq 0 \,\,\Longrightarrow\,\, {\rm tr}_{\hat g}g\leq C \,;$

\vspace{0.1cm}
\item[iii.] ${\rm tr}_{ \hat g} g\leq K\,,\quad  {\rm tr}_{ g}(\dot g+\widetilde{\rm Ric})\geq 0\,,\quad  \hat {\rm Ric}\leq 0 \,\,\Longrightarrow\,\, {\rm tr}_{ g} \hat g\leq C \,;$
\end{enumerate}
where $K\geq 0$ and $C$ depends on $M$, $K$, $g(0)$ and $\hat g$.
\end{lemma}
\begin{proof}
By using Lemma \ref{lemma3} and our assumptions we have 
$$
\begin{aligned}
\Box  \log\frac{\det g}{\det \hat  g}=&{\rm tr}_g(\dot g+\widetilde {\rm Ric})-{\rm tr}_{g}(\hat{\rm Ric})\leq K+ C {\rm tr}_{g}\hat g\leq C\,,
\end{aligned}
$$
and 
$$
\frac{\det g}{\det \hat  g}\leq {\rm e}^{Ct}\,,
$$
where $C$ depends on on $M$, $K$, $g(0)$ and $\hat g$. Since  
$$
{\rm tr}_{\hat g}g\leq \frac{1}{(n-1)!}({\rm tr}_{g} \hat g)^{n-1}\frac{\det g}{\det \hat g}\,,
$$
i) follows.  ii) and iii) easily follow from Lemma \ref{lemma3}.  
\end{proof}

For example, if $g(t)$ satisfies  the pluriclosed flow, then  $\dot g+\widetilde {\rm Ric}\geq 0$, while if $g(t)$ solves the positive Hermitian curvature flow, then $\dot g+\widetilde {\rm Ric}\leq 0$.  

\begin{proof}[Proof of Theorem $\ref{ThHCF}$]
If $g=g(t)$ is a solution to a flow belonging to the HCFs family on a compact complex manifold. Then 
$$
\dot g=-\widetilde {\rm Ric}+T*T
$$
implies 
\begin{equation}\label{111}
|\dot  g+\widetilde {\rm Ric}|_g=|T*T|_g\leq C |T|_g^2
\end{equation}
and 
\begin{equation}\label{222}
|\nabla(\dot  g+\widetilde {\rm Ric})|_g=|\nabla T*T|_g \leq C (|\nabla \Psi|_g+S^{1/2}+1) |T|_g\\
\leq C (|\nabla \Psi|_g+S+1) |T|_g\,. 
\end{equation}
Assume $g(t)$ defined for $t\in [0,\tau)$ and 
$$
\limsup_{t\to \tau}\,\max \{|{\rm tr}_{g}\hat g|_{C^0(g)},|T|_{C^0(g)}\}=K
$$
for some $K<\infty$. Then 
$$
{\rm tr}_{g}\hat g\leq K\mbox{ and }\,\, |T|_{g}\leq K	\,,\quad \mbox{ for all }t\in [0,\tau)\,. 
$$
Since 
$$
{\rm tr}_{g}(\dot g+\widetilde{\rm Ric})={\rm tr}_{g}(T*T)\leq C|T|_g^2\leq C
$$
where $C$ depends on $K$, then assumptions of Lemma \ref{61} are satisfied and 
$$
{\rm tr}_{\hat g}g\leq C {\rm e}^{Ct}\,,
$$
where $C$ depends on $K$, $g(0)$ and $\hat g$. It follows that 
$$
{\rm tr}_{\hat g}g\leq C {\rm e}^{C\tau}\mbox{ for }t\in [0,\tau)
$$
and 
$$
C^{-1}\hat g\leq g\leq C g\,.
$$
Moreover, the bound on $|T|_g$ together \eqref{111} and \eqref{222} implies that all the assumption of Theorem \ref{main1} are satisfied and, consequently, $g$ satisfies a $C^1$-a priori bound. Since $g$ satisfies a second order quasi-linear equation, standard theory implies that $g$ extends smoothly past time $\tau$  (see \cite[Section 4]{estimates}), contradicting the maximality of $[0,\tau)$. 
\end{proof}

Next we focus on the second Chern-Ricci flow 
$$
\dot g=-\widetilde{\rm Ric} 
$$
and we prove in particular Theorem \ref{ThH=0}. We consider the following

\begin{prop}
Assume that $g=g(t)$ solves the second Chern-Ricci flow on a compact Hermitian manifold $(M,\hat g)$. 
Then 
\begin{enumerate}
\item[i)] $\hat {\rm Ric } \geq  0$ implies $\omega^n\leq C \hat \omega^n$;

\vspace{0.1cm}
\item[ii)] $\hat {\rm Ric } =  0$ implies $\frac{1}{C}\hat \omega^n\leq \omega^n\leq C \hat \omega^n$;

\vspace{0.1cm}
\item[iii)] ${\rm Ric }(g)=0$ at $t=0$ implies that ${\rm Ric }(g)=0$ is zero for every $t$ and that the volume form $\omega^n$ is constant;

\vspace{0.1cm}
\item[iv)] $\hat g$ Chern-flat implies $C^{-1}\hat g\leq g\leq C\hat g$;
\end{enumerate}
where in each case $C$ is constant depending on $\hat g$ and $g(0)$. 
\end{prop}

\begin{proof}
For a solution to the second Chern-Ricci flow we have 
$$
\Box  \log\frac{\det g}{\det \hat g}=-{\rm tr}_{g}(\hat{\rm Ric})\,.
$$
Hence ${\rm tr}_{g}(\hat{\rm Ric})\geq 0$ implies $\Box  \log\frac{\det g}{\det \hat g}\leq 0$ and $i)$ follows. Moroever, 
if $\hat {\rm Ric} =0$, then 
$$
\Box  \log\frac{\det g}{\det \hat g}=0
$$
and ii) follows. 

If $g(0)$ is Chern-Ricci flat, we can take $\hat g=g_{0}$ and we have 
$$
\Box  \log\frac{\det g}{\det g(0)}=0\,\,,\quad \log\frac{\det g}{\det g(0)}\,_{|t=0}=0
$$
which implies that $\frac{\det g(t)}{\det g(0)}=1$ for every $t$ and iii) follows. 
\end{proof} 

\medskip
We have the following 
 
\begin{prop}\label{ultimo}
Let $(M,\hat{g})$ be a compact Hermitian manifold and assume that $\hat g$ has nonpositive holomorphic bisectional curvature. Let $g=g(t)$ be a maximal time solution to the second Chern-Ricci flow. 
If $|T|_g$ is bounded, then $g(t)$ is defined for $t\in[0,\infty)$. 
\end{prop}

\begin{proof}
If $\hat g$ has nonpositive holomorphic bisectional curvature, then Lemma \ref{lemma2} implies that ${\rm tr}_{g} \hat g$ is bounded and Theorem \ref{ThHCF} implies the claim. 
\end{proof}

Proposition \ref{ultimo} was proved in \cite{BV} under the strongest assumption that $M$ is the compact quotient of a Lie group by a lattice (and then it has a Chern-flat metric).   

\medskip 
Now we prove Theorem \ref{ThH=0}: 

\begin{proof}[Proof of Theorem $\ref{ThH=0}$]
In  view of \cite[Theorem 1.1]{HCF}, if $g(t)$, $t\in [0,\tau)$, is a maximal-time solution to a flow belonging to the family of HCF's and $\tau<\infty$, then  
$$
\limsup_{t\to \tau}\max\{ \|{\rm Rm}\|_{C^0(g)}, \|\nabla T\|_{C^0(g(t))}, \|T\|^2_{C^0(g(t))}\}=\infty\,.
$$
Let $g=g(t)$ be a solution to the second Chern-Ricci flow. We show that a bound on $|{\rm Rm}|_{g}$ implies also a bound on  $|T|_{g}$ and $|\nabla T|_g$. Indeed, if $|{\rm Rm}|_g$ is bounded, Lemma \ref{boxT} implies 
\begin{equation*}
\square\,|T|_g^2\leq C|T|_g^2\,,
\end{equation*}
where $C$ depends on the complex dimension of $M$,  
and the estimate on $|T|_g$ follows. 
Moreover, Lemma \ref{evcurv} implies   
\begin{equation*}
\square \,|\operatorname{Rm}|_g^2\leq -\frac{1}{2}|\nabla\operatorname{Rm}|_g^2-\frac{1}{2}|\bar\nabla\operatorname{Rm}|_g^2+C\,.
\end{equation*}
and from \cite{HCF} it follows  
$$
\square\, |\nabla T|_g^2\leq -|\nabla T|_g^2+\frac12 |\nabla {\rm Rm}|_g^2+C(|\nabla T|_g^2+1)\,. 
$$
Therefore we have  
$$
\begin{aligned}
 \square \left(|\operatorname{Rm}|_g^2+ |\nabla T|_g^2\right)\leq &-\frac{1}{2}|\nabla\operatorname{Rm}|_g^2-\frac{1}{2}|\bar\nabla\operatorname{Rm}|_g^2
 -|\nabla T|_g^2+\frac12 |\nabla {\rm Rm}|_g^2+C(|\nabla T|_g^2+1)\\
 \leq & C\left(|\nabla T|_g^2+|\operatorname{Rm}|_g^2\right)
\end{aligned}
$$
and the maximum principle implies the claim. 
\end{proof}

\section{Calabi estimate for a larger class of flows}

In the present section we focus on geometric flows of Hermitian metrics satisfying condition \eqref{moregeneral}. We briefly recall some flows introduced in the literature which fit in this class:  
 
\begin{itemize}
\item the Chern-Ricci flow belongs to this class of flows since in view of \eqref{R-R} if $g$ solves the Chern-Ricci, then  
$$
\dot g_{i\bar j}=-R_{i\bar j}=-\tilde{R}_{i\bar j}-g^{\bar lk}\left(\nabla_{\bar l}T_{ki\bar j}-\nabla_iT_{\bar l\bar jk}\right)\,;
$$

\vspace{0.1cm}
\item more generally, we can consider the geometric flows of the form 
$$
\partial_tg=-a{\rm Ric}+(a-1)\widetilde{\rm Ric}
$$
introduced in \cite{Liu}, where $a\in\R$ is a fixed parameter; 

\vspace{0.1cm}
\item another interesting class of flows was introduced by Boling \cite{Boling} and is governed by the equation 
$$
\partial_{t}g=-{\rm Ric}_s^{1,1}\,,
$$
where ${\rm Ric}_s$ is the Ricci form of the {\em Gauduchon connection} $\nabla^s=s\nabla+(1-s)\nabla^b$ and $\nabla^b$ denotes the 
{\em Bismut connection} \cite{Bismut}.  For $s=1$ the flow reduces to the Chern-Ricci flow, while for $s=0$ and $g$ pluriclosed the flow reduces to the pluriclosed flow; 

\vspace{0.1cm}
\item finally we consider the anti-complexified Chern-Ricci flow \cite{BGV}. This flow is defined on hypercomplex manifolds $(M,I,J,K)$ as
$$
\partial_tg=-\frac12({\rm Ric}+J{\rm Ric})
$$ 
where $g(t)$ is hyperhermitian for every $t$ and ${\rm Ric}$ is computed with respect to $g$ and $I$.  
\end{itemize} 
 
\begin{proof}[Proof of Proposition $\ref{2cor}$]
 Since \begin{equation*}
    \bar\nabla T=\bar\Psi*T+\bar{\hat \nabla} T\,,
\end{equation*}
we have 
\begin{equation*}
|\dot g+\tilde R|_g \leq C|\bar\nabla T|_g+C|T|^2_g\leq C(S^{1/2}+1)\,. 
\end{equation*} Furthermore,
\begin{equation*}
\nabla(\dot g+\tilde R)=\nabla\bar\nabla T+\nabla T*T\,.
\end{equation*} 
By using \eqref{terza} and \eqref{quinta} we get
\begin{equation*}
\begin{split}
|\nabla(\dot g+\tilde R)|_g\leq& C|\nabla\bar\nabla T|_g+C|\nabla T|_g|T|_g \\ \leq & C\left( |\bar\nabla\Psi|_g|T|_g+S^{1/2}|\hat\nabla T|_g+S|T|_g+|\hat\nabla\bar{\hat\nabla}T|_g+|\nabla\Psi|_g+S^{1/2}+1\right) \\
\leq &C\left(|\nabla\Psi|+|\bar\nabla\Psi|+S+S^{1/2}|\hat\nabla T|_g+1\right)\,.
\end{split}
    \end{equation*}Finally, we observe that since $T=g*\partial\omega$, we have
    \begin{equation*}
        \hat\nabla T=\partial T+\hat\Gamma*T=\partial\omega*\partial g+\hat\Gamma*T=\partial\omega*\Psi-\partial\omega*\hat\Gamma*g+\hat\Gamma*T\,,
    \end{equation*}hence
    \begin{equation*}
        |\hat\nabla T|_g\leq C(S^{1/2}+1)
    \end{equation*} and 
    \begin{equation*}
        S^{1/2}|\hat\nabla T|_g\leq CS\,.
    \end{equation*}
The claim follows. 
\end{proof}

\end{document}